\documentclass[pdflatex,sn-mathphys-num]{sn-jnl}

\usepackage{graphicx}%
\usepackage{multirow}%
\usepackage{amsmath,amssymb,amsfonts}%
\usepackage{amsthm}%
\usepackage{mathrsfs}%
\usepackage[title]{appendix}%
\usepackage{xcolor}%
\usepackage{textcomp}%
\usepackage{manyfoot}%
\usepackage{booktabs}%
\usepackage{algorithm}%
\usepackage{algorithmicx}%
\usepackage{algpseudocode}%
\usepackage{listings}%
\usepackage{enumitem}
\usepackage{tikz}
\usetikzlibrary{arrows.meta,calc,decorations.pathreplacing,shapes.misc}

\newcommand\diam{\mathop{\rm diam}}
\newcommand\dist{\mathop{\rm dist}}

\newcommand\argmin{\mathop{\rm argmin}}

\newcommand\gra{\mathop{\rm gra}}

\newcommand\R{\mathbb{R}}
\renewcommand\O{\mathcal{O}}

\newcommand\C{\mathcal{C}}
\newcommand\X{\mathcal{X}}
\newcommand\Y{\mathcal{Y}}
\newcommand\mF{\mathcal{F}}
\newcommand\tmF{\widetilde{\mF}}
\newcommand\dx{\dot{x}}
\newcommand\dw{\dot{w}}
\newcommand\mS{\mathcal{S}}
\renewcommand\L{\mathcal{L}}

\theoremstyle{thmstyleone}%
\newtheorem{theorem}{Theorem}
\newtheorem{lemma}[theorem]{Lemma}%
\newtheorem{corollary}[theorem]{Corollary}%

\theoremstyle{thmstyletwo}%
\newtheorem{remark}{Remark}%

\theoremstyle{thmstylethree}%
\newtheorem{definition}{Definition}%

\begin{document}

\title[Convergence of the Frank-Wolfe Algorithm for Monotone Variational Inequalities]{Convergence of the Frank-Wolfe Algorithm for Monotone Variational Inequalities}

\author*[1]{\fnm{Matthew} \sur{Hough}}\email{mhough@uwaterloo.ca}

\affil*[1]{\orgdiv{Department of Combinatorics and Optimization}, \orgname{University of Waterloo}, \orgaddress{\street{200 University Ave.~W.}, \city{Waterloo}, \postcode{N2L 3G1}, \state{Ontario}, \country{Canada}}}

\abstract{We consider the Frank-Wolfe algorithm for solving variational inequalities over compact, convex sets under a monotone
  $C^1$ operator and vanishing, nonsummable step sizes. We introduce a continuous-time interpolation of the discrete
  iteration and use tools from dynamical systems theory to analyze its asymptotic behavior. This allows us to
  derive convergence results for the original discrete algorithm. Consequently, every cluster point of the iterates
  is a solution of the underlying variational inequality, the distance from the iterates to the solution set converges
  to zero, and the Frank-Wolfe gap vanishes asymptotically. In the strongly monotone case, the solution is unique and
  the iterates converge to it. In particular, this proves Hammond's conjecture on the convergence of generalized
  fictitious play. We also discuss rates of convergence and under what assumptions rates can be shown.}

\keywords{Frank-Wolfe, conditional gradient method, projection-free, variational inequality problem}


\maketitle

\section{Introduction}\label{sec:intro}
We consider the variational inequality problem
\begin{equation} \tag{VIP} \label{prob:vip}
  \text{Find $x^* \in \C$ such that } \langle F(x^*), z-x^*\rangle \geq 0,
  \quad \forall z\in\C,
\end{equation}
and study the Frank-Wolfe iteration for solving \eqref{prob:vip}
\begin{equation} \label{alg:gfp} \tag{FW}
  x_{k+1} = x_k + \gamma_{k+1}(s_k - x_k),\quad s_k \in \beta(F(x_k)),\quad x_0 \in \C,
\end{equation}
where
\[
  \beta(u) := \argmin_{s\in\C}\langle{u, s}\rangle
\]
denotes the linear minimization oracle (LMO).
In general, we assume that $\C\subseteq\R^n$ is nonempty, compact, and convex, that $F:\C\to\R^n$ is monotone and $C^1$,
and that the stepsizes satisfy $\gamma_k\in(0,1]$, $\gamma_k\to 0$, and $\sum_{k=1}^\infty \gamma_k=\infty$.
This algorithm may be viewed as a generalization of the Frank-Wolfe algorithm \cite{FrankWolfe1956}
to the setting of monotone variational inequalities. It also contains generalized fictitious play (GFP)
\cite[Section~4.3.1]{Hammond1984} as the special case where $\gamma_k = 1/k$. Brown's classical fictitious
play algorithm (FP) \cite{Brown1951} is obtained as a further specialization when $\gamma_k = 1/k$, $\C = \Delta_n \times \Delta_m$ is a product of simplices,
and
\[
  F(x,y) = (-Ay, A^\top x),
\]
with $A \in \R^{n \times m}$ the payoff matrix.
We will denote the set of solutions to \eqref{prob:vip} by $\mS$, i.e.
\[
  \mS := \{x \in \C : \langle{F(x), z - x}\rangle \geq 0,\quad\forall z \in \C\}.
\]
To measure convergence to a solution of \eqref{prob:vip}, we introduce $V : \C \to \R_+$ the \textit{Frank-Wolfe gap}, defined by
\[
  V(x) := \max_{s \in \C}\langle{F(x), x - s}\rangle.
\]
We will see that $V$ is nonnegative on $\C$, and vanishes only on $\mS$. Moreover, for the special case of \eqref{alg:gfp}
corresponding to FP, $V$ is equivalent to the gap function used to measure convergence to a Nash equilibrium $(x^*,y^*)$ in the analysis of FP.
Indeed, $V$ is nonnegative for all $(x,y) \in \Delta_n\times\Delta_m$ and zero iff $(x,y)$ is a Nash equilibrium.

\paragraph{Related work.}
In the special case of FP, convergence was proven in \cite{Robinson1951}, which was later extended in \cite{Shapiro1958} to the convergence rate $\O(k^{-1/(m+n-2)})$.
Karlin later conjectured that the faster rate $\O(k^{-1/2})$ could be obtained, but this was refuted recently in \cite{Daskalakis2014} by
the construction of an adversarial tie-breaking strategy achieving a $\Omega(k^{-1/n})$ rate of convergence where $A$ is the $n\times n$ identity matrix.
After the discovery of this counterexample, the question of whether FP could still converge at Karlin's conjectured rate under a lexicographic tie-breaking rule
remained open. In 2021, it was proven in \cite{Abernethy2021} that for diagonal $A$, the convergence rate is indeed $\O(k^{-1/2})$. Notably, this class of $A$ includes
the identity matrix used in the counterexample \cite{Daskalakis2014}. However, the very recent result of Wang \cite{Wang2025} showed that the weaker form of Karlin's conjecture,
where ties are assumed to be broken lexicographically, was indeed false. Wang constructed a $10\times 10$ matrix for which FP converges at $\Omega(k^{-1/3})$ and
no ties occur except at the first step.

Despite the rich literature for FP, no convergence results are known for \eqref{alg:gfp} without additional assumptions on $\C$, even if $F$ is taken to be strongly monotone instead of just monotone.
Hammond conjectured the following for GFP in \cite[Section~4.3.1]{Hammond1984}:
\begin{quote}
  \emph{If $F$ is strongly monotone and $\C$ is a polytope, then generalized fictitious play will solve \eqref{prob:vip}.}
\end{quote}
To the best of our knowledge, prior to this work Hammond's conjecture remained open.

Note that when $F$ is only assumed to be monotone and not strongly monotone, \eqref{alg:gfp} can converge no faster than
FP, in general, for the step size $\gamma_k = 1/k$. It is unknown whether this statement can be extended to other choices of $\gamma_k$.

Variants of \eqref{alg:gfp} have drawn interest recently in the fields of optimization and computer science due to the projection-free nature of the iterations.
In particular, \cite{Gidel2017} introduce SP-FW, which considers the case of $\C = \X \times \Y$ and $F(x,y) = \left(\nabla_x \L(x,y), -\nabla_y \L(x,y)\right)$, where $\L$ is assumed to
be smooth and convex-concave. Since the subgradient of a convex function is monotone, it follows that $F(x,y)$ in this case is monotone.
The authors are able to obtain some convergence results for specific step sizes, but under strong assumptions such as strong convexity of $\X$ and $\Y$,
or uniform strong convex-concavity of $\L$ in addition to $\X$ and $\Y$ being polytopes with a very restrictive pyramidal width\footnote{Pyramidal width was first introduced in \cite{LacosteJ2015}} bound. Note that the uniform
strong convex-concavity of $\L$ implies $F$ is strongly monotone. Another related work is \cite{Hough2024}, in which the authors consider using iterations similar
to \eqref{alg:gfp} to solve linear programs. They call their algorithm FWLP, which on each iteration performs the following
\[
  \begin{cases}
  r_{k} \in \argmin_{r \in \Delta} \langle{c - A^Ty_k,r}\rangle,\\
  x_{k+1} = x_k + \gamma_{k+1}(r_k - x_k),\\
  s_{k} \in \argmin_{s \in \Gamma} \langle{b - Ax_{k+1},s}\rangle,\\
  y_{k+1} = y_k + \gamma_{k+1}(s_k - y_k),
  \end{cases}
\]
where $\Delta,\Gamma$ are polytopes and $\gamma_k = 1/k$. Interestingly, the $s_k$ update of FWLP relies on $x_{k+1}$
instead of $x_k$. If in the $s_k$ update, $x_{k+1}$ was replaced by $x_k$, FWLP could be rewritten completely in
the framework of \eqref{alg:gfp} with monotone operator $F(x,y) = (c - A^Ty, Ax - b)$. The authors in \cite{Hough2024} were unable to prove convergence
of FWLP.

\paragraph{Contributions.}
We prove that \eqref{alg:gfp} converges asymptotically to a solution of \eqref{prob:vip} in the sense that the Frank-Wolfe gap $V(x_k) \to 0$. Moreover,
if $F$ is additionally assumed to be strongly monotone, we show that $x_k \to x^* \in \mS$. Hammond's conjecture applies to the special case
where $\gamma_{k} = 1/k$, hence our result proves Hammond's conjecture. We also provide convergence rates in cases
where $\C$ is assumed to be strongly convex.

\subsection{Preliminaries}
Since $\C$ is compact, its diameter is well-defined. We denote it by $\diam(\C) := \max_{x,y \in \C} \|x-y\|$. Throughout,
$\|\cdot\|$ denotes the Euclidean norm, and $B(c,r)$ denotes the closed ball centered at $c \in \R^n$ with radius $r > 0$.

The following are standard definitions that we will use throughout.
\begin{definition}
  An operator $F : \C \to \R^n$ is called monotone if for all $x,y \in \C$,
  \[
    \langle{F(x) - F(y), x - y}\rangle \geq 0.
  \]
\end{definition}

\begin{definition}
  An operator $F : \C \to \R^n$ is called $\mu$-strongly monotone for some $\mu > 0$ if for all $x,y \in \C$,
  \[
    \langle{F(x) - F(y), x - y}\rangle \geq \mu\|x-y\|^2.
  \]
\end{definition}

\begin{definition}
  An operator $F : \C \to \R^n$ is called $\beta$-cocoercive for some $\beta > 0$ if for all $x,y \in \C$,
  \[
    \langle{F(x) - F(y), x - y}\rangle \geq \beta\|F(x) - F(y)\|^2.
  \]
\end{definition}

\begin{definition}
  A nonempty set $C \subseteq \R^n$ is $\alpha$-strongly convex if for every $x,y\in C$ and every $t \in [0,1]$,
  \[
    B\left((1-t)x + ty, \frac{\alpha}{2}t(1-t)\lVert x - y\rVert^2\right) \subseteq C.
  \]
\end{definition}

\begin{definition} \label{def:soln}
  Let $I \subseteq \R$ be an interval and let $G : \R^n \rightrightarrows \R^n$.
  A map $x : I \to \R^n$ is called a solution of the differential inclusion
  \[
    \dot x(t) \in G(x(t)),
  \]
  if $x$ is absolutely continuous and
  \[
    \dx(t) \in G(x(t)), \qquad \text{for a.e. } t \in I.
  \]
  All solutions in this paper are understood in this sense.
\end{definition}

\section{Asymptotic convergence to solutions of the VIP} \label{sec:hammond}
Throughout this section, we assume that $F:\C \to \R^n$ is monotone and $C^1$, that $\C\subseteq\R^n$ is nonempty, compact, and convex,
and that $\{x_k\}$ is the sequence generated by \eqref{alg:gfp}.

\begin{lemma} \label{lem:sol-nonempty}
   The solution set $\mS$ is nonempty.
\end{lemma}
\begin{proof}
  This follows from the Hartman-Stampacchia theorem \cite[Lemma~3.1]{Hartman1966} and the assumption that $F$ is $C^1$ and $\C$ is nonempty, convex, and compact.
\end{proof}

Define the following set-valued map on $\C$:
\[
  \mF(x) := \beta(F(x)) - x,\qquad x \in \C.
\]
To analyze \eqref{alg:gfp}, we introduce the differential inclusion
\begin{equation}\label{eqn:BR}\tag{DI-$\C$}
  \dot x(t)\in \mF(x(t)) = \beta(F(x(t))) - x(t),\qquad x(t) \in \C.
\end{equation}
This is the continuous-time counterpart of the Frank-Wolfe iteration, since
\[
  \frac{x_{k+1}-x_k}{\gamma_{k+1}} = s_k-x_k \in \beta(F(x_k)) - x_k = \mF(x_k).
\]
Thus, if the step size $\gamma_{k+1}$ is viewed as a small time increment, the discrete update formally approaches the inclusion above.
For each $x\in\C$, the set $\mF(x)=\beta(F(x))-x$ consists of all vectors of the form $s-x$ with $s\in\beta(F(x))$. Thus, the differential inclusion
\eqref{eqn:BR}
means that, at each time $t$, the rate of change of the trajectory $x(t)$ is given by the vector from the current point $x(t)$ to some current solution $s(t)\in \beta(F(x(t)))$ of the linear minimization oracle.

However, $\mF$ is defined only for points in $\C$, while the results of \cite{Benaim2005} we appeal to are stated for inclusions on all of $\R^n$. We therefore introduce the projected extension (cf. \cite[Remark~1.2]{Benaim2005})
\[
  \tmF(x) := \beta(F(P(x))) - x = \big\{s - x : s \in \beta(F(P(x)))\big\},
\]
where $P:\R^n\to\C$ is the Euclidean projection. Since $P(x)=x$ for every $x\in\C$, the extension agrees with $\mF$ on $\C$:
\[
  \tmF(x)=\mF(x),\quad x\in\C.
\]
Accordingly, we will study the following differential inclusion defined on the whole space:
\begin{equation} \label{prob:pdi} \tag{DI}
  \dot x(t)\in \tmF(x(t)) = \beta(F(P(x(t)))) - x(t), \qquad t \in \R.
\end{equation}

\begin{lemma}\label{lem:Chen}
Recall the gap function $V:\C\to\R_+$ given by
\[
  V(x):= \langle{F(x),x}\rangle -\min_{s\in\C}\langle{F(x),s}\rangle = \max_{s \in \C}\langle{F(x), x - s}\rangle.
\]
The following hold.
\begin{enumerate}[label=(\alph*)]
\item $V$ is Lipschitz continuous on $\C$, $V(x)\geq 0$ for all $x\in\C$, and $V^{-1}(0)=\mS$. Moreover, $\mS$ is closed.
\item Let $x:[0,\infty)\to\C$ be any solution of \eqref{eqn:BR}.
Then $t\mapsto V(x(t))$ satisfies
\[
  \frac{d}{dt}V(x(t)) \leq -V(x(t))\quad \text{for a.e. }t\geq 0.
\]
Consequently,
\[
  V(x(t))\leq e^{-t}V(x(0))\quad\forall t\geq 0.
\]
In particular, if $x(0)\notin \mS$ then $V(x(t)) < V(x(0))$ for every $t>0$,
and if $x(0)\in\mS$ then $V(x(t)) = 0$ for all $t\ge 0$.
\end{enumerate}
\end{lemma}

\begin{proof}
(a) The fact that $V$ is Lipschitz on $\C$ follows from \cite[Theorem~1]{Chen2024} with the assumption that $F$ is $C^1$. Nonnegativity is clear from the definition of $V$ and the fact that $x \in \C$.
$V(x)=0$ iff $x \in \mS$, because for any $x \in \C$
\[
  \max_{s \in \C} \langle{F(x), x - s}\rangle = 0 \iff \langle{F(x), x - z}\rangle\leq 0\quad\forall z \in \C \iff \langle{F(x), z - x}\rangle \geq 0\quad\forall z \in \C
\]
To show $\mS$ is closed, we note that since $\mS = V^{-1}(\{0\})$ and $V$ is Lipschitz continuous,
$\mS$ is the preimage of a closed set under a continuous map and therefore must be closed in $\C$. Because $\C$ is closed in $\R^n$,
$\mS$ must be closed in $\R^n$.

(b) The differential inequality comes from the proof of \cite[Theorem~1]{Chen2024}. Let $x$ be a solution of \eqref{eqn:BR}. To obtain the bound on $V(x(t))$ for all $t \geq 0$, fix some $T > 0$. Recall from (a) 
that $V$ is Lipschitz continuous, and $t \mapsto x(t)$ is absolutely continuous by Definition~\ref{def:soln}. Hence, their composition
$V(x(t))$ is absolutely continuous on $[0,T]$. We proceed by showing a Gr\"onwall-type inequality inspired by the proof of \cite[Theorem~1.12]{Tao2006}.
Let $v(t) := V(x(t))$ and $u(t) := v(t)e^t$, then $u(t)$ is absolutely continuous on $[0,T]$, since $v(t)$ is. Moreover,
since $v'(t) \leq -v(t)$ for a.e. $t \geq 0$,
\[
  u'(t) = v'(t) e^t + v(t)e^t \leq -v(t)e^t + v(t)e^t = 0,
\]
for a.e. $t \in [0,T]$, which along with the absolute continuity of $u(t)$, implies that for any $0 \leq r \leq t \leq T$,
\[
  u(t) - u(r) = \int_r^t u'(s)\ ds \leq 0.
\]
Hence, $u(t)$ is nonincreasing on $[0,T]$, and in particular $u(t) \leq u(0)$ for all $t \in [0,T]$. This can be rewritten as
\[
  v(t)e^t \leq v(0)e^0 = v(0),\qquad \forall t \in [0,T],
\]
which after rearranging gives
\[
  V(x(t))\leq e^{-t}V(x(0))\quad\forall t \in [0,T].
\]
Since $T > 0$ was arbitrary, the above inequality holds for all $t \geq 0$.
\end{proof}

We now construct an interpolating curve for the algorithm \eqref{alg:gfp}. Set $\tau_0 := 0$ and
\[
  \tau_{k+1} := \tau_k + \gamma_{k+1} = \sum_{i=1}^{k+1}\gamma_i.
\]
Since $\sum_{k=1}^{\infty}\gamma_k = \infty$, it follows that $\tau_k \to \infty$ as $k \to \infty$. 
Suppose $\{x_k\}$ is the sequence of iterates generated by \eqref{alg:gfp} under the assumptions of Section~\ref{sec:intro}.
Define the interpolating curve $w : [0,\infty) \to \C$
so that $w(\tau_k) = x_k$ for each $k$ and it is linear on the time segments $t \in [\tau_k,\tau_{k+1}]$:
\[
  w(t) := (1-\theta(t))x_k + \theta(t)x_{k+1},\quad t \in [\tau_k,\tau_{k+1}],
\]
where
\[
  \theta(t) = \frac{t-\tau_k}{\gamma_{k+1}} = \frac{t - \tau_k}{\tau_{k+1} - \tau_k},\quad t \in [\tau_k,\tau_{k+1}].
\]
Clearly, if $t = \tau_k$ $w(t) = x_k$ and if $t = \tau_{k+1}$, $w(t) = x_{k+1}$. Also,
\[
 0 = \frac{\tau_k - \tau_k}{\tau_{k+1} - \tau_k} \leq \frac{t - \tau_k}{\tau_{k+1} - \tau_k} \leq \frac{\tau_{k+1} - \tau_k}{\tau_{k+1} - \tau_k} = 1,
\]
so $\theta(t) \in [0,1]$. Therefore, for any $t$, $w(t)$ is a convex combination of points in $\C$, hence $w(t) \in \C$ for all $t$.
On each open interval $(\tau_k,\tau_{k+1})$, $w$ is linear and is written in full as
\[
  w(t) = t\frac{x_{k+1} - x_k}{\gamma_{k+1}} + \left(1 + \frac{\tau_k}{\gamma_{k+1}}\right)x_k - \frac{\tau_k}{\gamma_{k+1}}x_{k+1},
\]
with derivative
\[
  \dw(t) = \frac{x_{k+1} - x_k}{\gamma_{k+1}} = \frac{x_{k+1} - x_k}{\tau_{k+1} - \tau_k}.
\]
Since $x_{k+1} - x_k = \gamma_{k+1}(s_k - x_k)$, we have for $t \in (\tau_k,\tau_{k+1})$
\[
  \dw(t) = s_k - x_k \in \beta(F(x_k)) - x_k = \mF(x_k),
\]
where the derivative exists for all $t \geq 0$, except the breakpoints $t = \tau_k$ for all $k$. Hence, $\dw(t) \in \mF(x_k)$ a.e. on $t \geq 0$.
The next few results show that this fits into the framework of \cite{Benaim2005}.

\begin{lemma}\label{lem:BHS-H1.1}
The projected extension $\tmF$ satisfies the following:
\begin{enumerate}
  \item[(a)] $\tmF(x)$ is nonempty, compact, and convex for all $x\in\R^n$.
  \item[(b)] $\tmF$ has closed graph in $\R^n\times\R^n$.
  \item[(c)] There exists $c>0$ such that $\sup_{z\in\tmF(x)}\|z\|\leq c(1+\|x\|)$ for all $x\in\R^n$.
\end{enumerate}
\end{lemma}

\begin{proof}
Since $\C$ is compact and $s\mapsto\langle u,s\rangle$ is continuous and affine,
$\beta(u)$ is nonempty and compact for every $u$. It is also convex because it is the
set of minimizers of a linear functional over a convex set. Also, since $\C$ is compact and convex, $P$ is single-valued
and nonexpansive, so continuous.

(a) Fix $x\in\R^n$. Then $\beta(F(P(x)))$ is nonempty, compact, convex, and contained in $\C$.
Hence
\[
  \tmF(x)=\beta(F(P(x))) - x
\]
is a translation of a nonempty, compact, and convex set, so it is itself nonempty, compact, and convex.

(b) Let $x_k\to x$ in $\R^n$ and $y_k\to y$ in $\R^n$ with $y_k\in\tmF(x_k)$.
Then there exist $b_k\in \beta(F(P(x_k)))$ such that
\[
  y_k=b_k-x_k.
\]
Since $b_k\in\C$ and $\C$ is compact, we may pass to a subsequence such that $b_k\to b\in\C$.
Because $P$ and $F$ are continuous, $F(P(x_k))\to F(P(x))$.

Now fix any $s\in\C$. Optimality of $b_k$ gives
\[
  \langle F(P(x_k)), b_k\rangle \leq \langle F(P(x_k)), s\rangle, \quad \forall k.
\]
Taking $k\to\infty$ yields
\[
  \langle F(P(x)), b\rangle \le \langle F(P(x)), s\rangle, \quad \forall s\in\C,
\]
so $b\in \beta(F(P(x)))$. Finally, $y_k=b_k-x_k\to b-x=:y$, hence
$y\in \beta(F(P(x))) - x = \tmF(x)$. Therefore $\gra(\tmF)$ is closed.

(c) Let $M:=\max_{u\in\C}\lVert u\rVert$ and note $M < \infty$ because $\C$ is compact. For any $z\in\tmF(x)$ we can write
$z=b-x$ with $b\in\C$, hence $\lVert z\rVert \leq \lVert b\rVert + \lVert x\rVert \leq M+\lVert x\rVert$.
Thus $\sup_{z\in\tmF(x)}\lVert z\rVert \leq M+\lVert x\rVert \leq c(1+\lVert x\rVert)$ with $c:=\max(M,1)$.
\end{proof}

\begin{remark}\label{rem:existence}
By Lemma~\ref{lem:BHS-H1.1}, the differential inclusion \eqref{prob:pdi} admits at least one solution $x : \R \to \R^n$ through every initial
condition $x(0) = x_0\in\R^n$ (see \cite[Section~1.2]{Benaim2005}).
\end{remark}

\begin{lemma} \label{lem:lipschitz-w}
  The interpolated curve $w$ is Lipschitz continuous.
\end{lemma}
\begin{proof}
  Recall, on $[\tau_k, \tau_{k+1}]$
  \[
    w(t) = (1-\theta(t))x_k + \theta(t)x_{k+1},
  \]
  so for any $t,t' \in [\tau_k,\tau_{k+1}]$,
  \[
    w(t) - w(t') = \theta(t')x_k - \theta(t')x_{k+1} - \theta(t) x_k + \theta(t)x_{k+1} = \left(\theta(t) - \theta(t')\right)(x_{k+1} - x_k).
  \]
  Hence,
  \[
    \lVert w(t) - w(t')\rVert = \lVert\left(\theta(t) - \theta(t')\right)(x_{k+1} - x_k)\rVert \leq \lvert\theta(t) - \theta(t')\rvert\lVert x_{k+1} - x_k\rVert.
  \]
  But $x_{k+1} - x_k = \gamma_{k+1}(s_k - x_k)$ and $x_k,s_k \in \C$, so $\lVert s_k - x_k\rVert \leq \diam(\C)$, and $\diam(\C) < \infty$ by compactness.
  Also,
  \[
    \theta(t) - \theta(t') = \frac{t - \tau_k}{\gamma_{k+1}} - \frac{t' - \tau_k}{\gamma_{k+1}} = \frac{t-t'}{\gamma_{k+1}}.
  \]
  Thus,
  \[
    \lVert w(t) - w(t')\rVert \leq \diam(\C)\lvert t - t'\rvert.
  \]
  Now consider the case of arbitrary $a,b \in \R_+$. We may assume wlog that $a < b$. If $a,b$ lie in the same interval, we are in the case of the above,
  so suppose $a \in [\tau_m,\tau_{m+1}]$ and $b \in [\tau_n, \tau_{n+1}]$ where $m \leq n$. We may write $w(b) - w(a)$ using a telescoping sum:
  \[
    w(b) - w(a) = \left(w(b) - w(\tau_n)\right) + \left(w(\tau_{m+1}) - w(a)\right) + \sum_{k=m+1}^{n-1}\left(w(\tau_{k+1}) - w(\tau_{k}))\right).
  \]
  By the triangle inequality and the above case when the two breakpoints are in the same interval, we have
  \begin{align*}
    \lVert w(b) - w(a)\rVert &\leq \lVert w(b) - w(\tau_n)\rVert + \lVert w(\tau_{m+1}) - w(a)\rVert + \sum_{k=m+1}^{n-1}\lVert w(\tau_{k+1}) - w(\tau_k)\rVert\\
                             &\leq \diam(\C)(b-\tau_n) + \diam(\C)(\tau_{m+1} - a) + \diam(\C)\sum_{k=m+1}^{n-1}(\tau_{k+1} - \tau_k)\\
                             &= \diam(\C)(b-\tau_n) + \diam(\C)(\tau_{m+1} - a) + \diam(\C)(\tau_n - \tau_{m+1})\\
                             &= \diam(\C)(b-a).
  \end{align*}
  It follows that $w$ is Lipschitz continuous.
\end{proof}

The interpolated curve $w$ is not, in general, an exact solution of the differential inclusion \eqref{prob:pdi}.
Nevertheless, because it is obtained by linearly interpolating the Frank-Wolfe iterates, it provides a natural
approximation of the continuous-time dynamics, up to a discretization error that vanishes asymptotically as
$\gamma_k\to 0$.
To formalize this idea, we use the notion of a \emph{perturbed solution} \cite[Definition~II]{Benaim2005}.
Roughly speaking, a perturbed solution is an absolutely continuous curve that satisfies the differential inclusion up to errors
that become negligible in the long run. This notion is useful because, even though a perturbed solution is only approximate,
we will see that its asymptotic behavior can still be analyzed through the corresponding differential inclusion.

\begin{definition} \label{def:perturbed-soln}
  Let $y : [0,\infty) \to \R^n$ be continuous. We say that $y$ is a perturbed solution of the differential inclusion \eqref{prob:pdi}
  if:
  \begin{enumerate}[label=(\roman*)]
    \item $y$ is absolutely continuous;
    \item there exist a locally integrable function $t \mapsto U(t)$ and $\delta:[0,\infty) \to \R_+$ with $\delta(t)\to 0$ as $t\to\infty$ such that:
    \begin{enumerate}[label=(\alph*)]
      \item for all $T>0$,
      \[
        \lim_{t\to\infty}\sup_{0\leq v\leq T}\bigg\|\int_t^{t+v}U(s)\,ds\bigg\| = 0;
      \]
      \item $\dot y(t)-U(t)\in \tmF^{\delta(t)}(y(t))$
      for almost every $t>0$, where
      \[
        \tmF^\delta(x) := \big\{u\in\R^n : \exists z \in\R^n \text{ with } \|z-x\| < \delta \text{ and } \inf_{f \in \tmF(z)}\|u - f\| < \delta\big\}.
      \]
    \end{enumerate}
  \end{enumerate}
\end{definition}
Condition (ii)(a) says that the cumulative effect of the additive perturbation $U$ on every bounded time window becomes negligible as $t\to\infty$,
while condition (ii)(b) permits a vanishing approximation error both in the point at which the map is evaluated and in the inclusion itself.
In our setting, the interpolated curve $w$ satisfies this definition with $U \equiv 0$. Thus the only discrepancy from being an exact solution of \eqref{prob:pdi} is that, on each interval $(\tau_k,\tau_{k+1})$, one has
\[
  \dot w(t)=s_k-x_k\in \tmF(x_k),
\]
whereas an exact solution would require
\[
  \dot w(t)\in \tmF(w(t)).
\]
Since $w(t)$ remains close to $x_k$ on this interval and the distance tends to zero as $k\to\infty$, this discrepancy is asymptotically negligible.

\begin{lemma} \label{lem:perturbed}
  The interpolated curve $w$ is a perturbed solution of the differential inclusion \eqref{prob:pdi}.
\end{lemma}
\begin{proof}
  By definition of $w$, it is a continuous function from $[0,\infty)$ to $\C \subseteq \R^n$. Moreover, since $w$ is Lipschitz continuous
  from Lemma~\ref{lem:lipschitz-w}, it is absolutely continuous. It remains to satisfy the conditions (ii) in Definition~\ref{def:perturbed-soln}.
  In Definition~\ref{def:perturbed-soln}, take $U \equiv 0$, then (ii)(a) must be satisfied. To satisfy (ii)(b), we need to show $\dw(t) \in \tmF^{\delta(t)}(w(t))$
  for almost every $t > 0$. Define $\delta(t) = \gamma_{k+1}(1 + \diam(\C))$ for each $t \in [\tau_k,\tau_{k+1})$, then clearly $\delta : [0,\infty) \to \R$ and $\delta(t) \to 0$.
  Recall
  \[
    \tmF^{\delta}(x) = \{u \in \R^n : \exists z \text{ with } \lVert z - x\rVert < \delta,\, \inf_{f \in \tmF(z)}\lVert u - f\rVert < \delta\}.
  \]
  We have already shown that $\dw(t) \in \mF(x_k) = \tmF(x_k)$, where the equality comes from $x_k \in \C$. Thus, if we take $z = x_k$ in the definition of $\tmF^{\delta(t)}(w(t))$, we have
  \[
    \inf_{f \in \tmF(x_k)} \lVert \dw(t) - f\rVert = 0 < \delta(t).
  \]
  Moreover, for $t \in (\tau_k,\tau_{k+1})$,
  \[
    w(t) - x_k = \theta(t)(x_{k+1} - x_k),
  \]
  and $\theta(t) \in (0,1)$, so
  \[
    \lVert w(t) - x_k\rVert = \theta(t)\lVert x_{k+1} - x_k\rVert \leq \lVert x_{k+1} - x_k\rVert = \gamma_{k+1}\lVert s_k - x_k\rVert \leq \gamma_{k+1}\diam(\C) < \delta(t).
  \]
  We have shown that $\dw(t) \in \tmF^{\delta(t)}(w(t))$ for all $t$ in the open segments $(\tau_k,\tau_{k+1})$. Since the endpoints form a countable set,
  we have shown that $\dw(t) \in \tmF^{\delta(t)}(w(t))$ for almost every $t > 0$, hence we have satisfied condition (ii)(b).
\end{proof}

Recall the definition of the limit set of $w$,
\[
  L(w) := \bigcap_{t \geq 0} \overline{\{w(s) : s \geq t\}}.
\]
Our goal will be to show that $L(w) \subseteq \mS$, and then conclude that every cluster point of $\{x_k\}$ belongs to $\mS$.
To do so, we use the notion of invariance of a set with respect to the differential inclusion \eqref{prob:pdi}.

\begin{definition}
  A set $A \subseteq \R^n$ is said to be invariant if for all $z \in A$ there exists a solution $x$ to \eqref{prob:pdi} with
  $x(0) = z$ where $x(\R) \subseteq A$.
\end{definition}

\begin{theorem}\label{thm:limitset}
  $L(w) \subseteq \mS$.
\end{theorem}
\begin{proof}
Since $w([0,\infty))\subseteq \C$ and $\C$ is compact, we have $L(w)\subseteq \C$ and $L(w)$ is nonempty.

By Lemma~\ref{lem:perturbed}, $w$ is a bounded perturbed solution of the inclusion
$\dot x(t) \in \tmF(x(t))$, hence by \cite[Theorem~3.6]{Benaim2005} the set $L(w)$ is internally
chain transitive \cite[Definition VI]{Benaim2005} for the dynamical system generated by the differential inclusion \eqref{prob:pdi}. Therefore, by \cite[Lemma~3.5]{Benaim2005}, $L(w)$ is invariant. That is, for every
$z\in L(w)$ there exists a solution $x:\R\to\R^n$ of $\dot x(t)\in\tmF(x(t))$ with
$x(0)=z$ and $x(\R)\subseteq L(w)\subseteq \C$. Fix $z\in L(w)$ and let $x$ be such a solution in $L(w)$ with $x(0)=z$.
Because $x(\R)\subseteq\C$, we have $\tmF(x(t))=\mF(x(t))$ for all $t \in \R$, hence $x$ is also
a solution of the differential inclusion \eqref{eqn:BR} for a.e. $t \in \R$.

Let $M:=\max_{u\in \C} V(u)<\infty$, where the finiteness comes from the compactness of $\C$ and continuity of $V$ (from $V$ Lipschitz, cf. Lemma~\ref{lem:Chen}(a)).
For any $T>0$, define $y:[0,\infty)\to\C$ by $y(t):=x(t-T)$. Since $x$ is a solution of \eqref{eqn:BR} on $\R$, it follows that $y$ is a solution of
\eqref{eqn:BR} on $[0,\infty)$. Hence, we may apply Lemma~\ref{lem:Chen}(b) to $y$:
\[
  V(y(t)) \leq e^{-t} V(y(0)),\qquad\forall t \geq 0,
\]
which when taking $t = T$ is equivalent to
\[
  V(x(0)) \leq e^{-T}V(x(-T)).
\]
Since $x(-T)\in\C$, we must have $V(x(-T))\leq M$, so we may write
\[
V(z) = V(x(0)) \leq e^{-T} M, \qquad \forall\ T>0.
\]
Letting $T\to\infty$ gives $V(z)=0$. By Lemma~\ref{lem:Chen}(a), $V^{-1}(0)=\mS$, so
$z\in\mS$.
Thus $L(w)\subseteq \mS$.
\end{proof}

\begin{corollary} \label{cor:cluster}
Every cluster point of the sequence $\{x_k\}$ belongs to $\mS$. In particular, as $k \to \infty$,
\[
  \dist(x_k,\mS) \to 0
\]
\end{corollary}

\begin{proof}
Let $\bar x$ be a cluster point of $\{x_k\}$. Then there exists a subsequence
$\{x_{k_j}\}$ such that $x_{k_j} \to \bar{x}$ as $j \to \infty$.
Since $x_k = w(\tau_k)$ for all $k$ and $\tau_k \to \infty$, we also have
$\tau_{k_j} \to \infty$ and $w(\tau_{k_j}) = x_{k_j} \to \bar{x}$.
We claim that $\bar x \in L(w)$. Indeed, fix any $t\ge 0$. Since $\tau_{k_j}\to\infty$,
there exists $j_0$ such that $\tau_{k_j}\ge t$ for all $j\ge j_0$. Hence
\[
  w(\tau_{k_j}) \in \{w(s): s\geq t\}, \qquad \forall j\geq j_0.
\]
Taking the limit and using $w(\tau_{k_j})\to \bar x$, we obtain
\[
  \bar x \in \overline{\{w(s): s\geq t\}}.
\]
Since $t\geq 0$ was arbitrary, it follows that
\[
  \bar x \in \bigcap_{t\ge 0} \overline{\{w(s): s\ge t\}} = L(w).
\]
By Theorem~\ref{thm:limitset}, $L(w)\subseteq \mS$, so $\bar x \in \mS$.
This proves the first claim.

For the second claim, suppose for contradiction that $\dist(x_k,\mS) \not\to 0$.
Then there exist $\epsilon >0$ and a subsequence $\{x_{k_j}\}$ such that
\[
  \dist(x_{k_j},\mS) \geq \epsilon, \qquad \forall j.
\]
Since $\{x_k\}\subseteq \C$ and $\C$ is compact, by passing to a further subsequence if
necessary, we have that $x_{k_j}\to \bar x$ for some $\bar x\in\C$.
By the above, $\bar x\in \mS$. Since $\mS$ is closed by Lemma~\ref{lem:Chen}(a),
$x\mapsto \dist(x,\mS)$ is continuous, and therefore
\[
  \dist(x_{k_j},\mS) \to \dist(\bar x,\mS) = 0,
\]
which contradicts $\dist(x_{k_j},\mS) \geq \epsilon$ for all $j$.
Hence $\dist(x_k, \mS) \to 0$ as $k \to \infty$.
\end{proof}

\begin{corollary}
  The Frank-Wolfe gap along the iterates $\{x_k\}$ converges to zero:
  \[
    V(x_k) \to 0 \text{ as } k\to\infty.
  \]
\end{corollary}
\begin{proof}
  By Lemma~\ref{lem:Chen}(a), $V$ is Lipschitz continuous on $\C$, so there exists some $L>0$ such that
  \[
    \lvert V(x)-V(y)\rvert \leq L\lVert x-y\rVert, \qquad \forall x,y\in\C.
  \]
  Fix $x\in\C$. For any $z\in \mS$, Lemma~\ref{lem:Chen}(a) gives
  $V(z)=0$, so
  \[
    V(x)=|V(x)-V(z)| \leq L\|x-z\|.
  \]
  Taking the infimum of the above over $z\in \mS$ yields
  \[
    V(x)\leq L\,\dist(x,\mS).
  \]
  Applying this with $x=x_k$ and using $\dist(x_k,\mS)\to 0$ gives $V(x_k)\to 0$.
\end{proof}

\subsection{The strongly monotone case}
Suppose in addition that $F$ is $\mu$-strongly monotone for some $\mu > 0$, i.e. for all $x,y\in \C$
\[
  \langle{F(x) - F(y), x - y}\rangle \geq \mu\lVert x - y\rVert^2.
\]
Under this additional assumption, the solution set is a singleton.
\begin{lemma} \label{lem:singleton}
  The solution set $\mS$ is a singleton.
\end{lemma}
\begin{proof}
  Suppose $x,y \in \mS$ are distinct. Then,
  \[
    \langle{F(x), y - x}\rangle \geq 0,\quad \langle{F(y), x - y}\rangle \geq 0.
  \]
  Adding the two gives
  \[
    \langle{F(x) - F(y), x - y}\rangle \leq 0.
  \]
  We may combine this with the definition of strong monotonicity to obtain
  \[
    0 \geq \langle{F(x) - F(y), x - y}\rangle \geq \mu\lVert x - y\rVert^2,
  \]
  which implies $x = y$. So $\mS$ must contain only one element.
\end{proof}

\begin{theorem}
  Let $x^*$ be the unique element in $\mS$. Then $x_k \to x^*$.
\end{theorem}
\begin{proof}
  By Corollary~\ref{cor:cluster}, $\dist(x_k,\mS) \to 0$.
  Since $\mS = \{x^*\}$, by Lemma~\ref{lem:singleton}, we have
  \[
    \lVert x_k-x^*\rVert = \dist(x_k,\{x^*\}) = \dist(x_k,\mS) \to 0.
  \]
  Hence $x_k\to x^*$.
\end{proof}

\section{Rates of convergence}
\subsection{On obtaining rates from the continuous-time analysis}
By Lemma~\ref{lem:Chen}(b), every solution $x$ of \eqref{eqn:BR} satisfies
\[
  V(x(t)) \leq e^{-t}V(x(0)), \qquad t \geq 0.
\]
Thus the Frank-Wolfe gap converges to zero at an exponential rate along trajectories of the continuous-time dynamics.
It is not clear, however, how to transfer this rate to the discrete iterates. The difficulty is that the interpolated
curve $w$ associated with the sequence $\{x_k\}$ is not, in general, a solution of the differential inclusion, but only a
perturbed solution. On each interval $(\tau_k,\tau_{k+1})$,
\[
  \dw(t)= s_k - x_k \in \tmF(x_k),
\]
whereas an exact solution would satisfy
\[
  \dw(t)\in \tmF(w(t)).
\]
To deduce a discrete convergence rate from the continuous-time analysis, one would need
to control the error introduced by replacing $\tmF(w(t))$ with $\tmF(x_k)$.

This is the main obstruction in the setting where $\C$ is a general compact, convex set.
Without additional geometric assumptions on $\C$, the LMO need not
depend continuously on its argument, so small changes in $F(x)$ may produce large changes in the selected minimizer. Hence, even though
$\|w(t)-x_k\| = \mathcal O(\gamma_{k+1})$
for $t \in [\tau_k,\tau_{k+1}]$, this does not imply that one can choose minimizers in $\beta(F(w(t)))$ and $\beta(F(x_k))$ that are close.
For this reason, the continuous-time rate does not directly yield a rate for the discrete algorithm.

In the next subsection, we show that strong convexity of $\C$ gives sufficient regularity of the LMO
to control this error and derive rates for the discrete iterates.

\subsection{Rates over strongly convex sets}
When $\C$ is assumed to be strongly convex in addition to being nonempty and compact, the linear minimization oracle becomes
Lipschitz over the unit sphere.
\begin{lemma} \label{lem:lmo-lipschitz}
  Let $\C \subseteq \mathbb{R}^n$ be nonempty, compact, and $\alpha$-strongly convex for some $\alpha > 0$. Recall the definition of the linear minimization oracle
  \[
    \beta(u) := \argmin_{s\in\C}\langle{u, s}\rangle,\quad u \in \mathbb{R}^n.
  \]
  Then,
  \begin{enumerate}[label=(\roman*)]
      \item For every $u \in \mathbb{R}^n\setminus\{0\}$, the minimizer $\beta(u)$ is unique.
      \item The function $\beta(u)$ is $(1/\alpha)$-Lipschitz on the unit sphere; that is, for all unit vectors $u,v \in \mathbb{R}^n$
      \begin{equation} \label{eq:lmo-lip}
          \lVert \beta(u) - \beta(v)\rVert \leq \frac{1}{\alpha}\lVert u - v\rVert.
      \end{equation}
  \end{enumerate}
\end{lemma}
\begin{proof}
  First, since $\C$ is nonempty and compact, $\beta(u)$ must exist. Now suppose $u \neq 0$ and there exist distinct $s_1,s_2 \in \C$ with
  $\langle{u,s_1}\rangle = \langle{u,s_2}\rangle = k \in \mathbb{R}$. Define $m = (s_1 + s_2)/2$. By $\alpha$-strong convexity of $\C$, for $t=1/2$ we have
  \[
    B(m,\rho) \subseteq \C,
  \]
  where $\rho = \frac{\alpha}{8}\lVert s_1 - s_2\rVert^2$, which is positive by the assumption that $s_1 \neq s_2$. Take $w = u/\lVert u\rVert$ and
  $z = m - \rho w$. We must have $z \in \C$, while
  \begin{align*}
    \langle{u,z}\rangle &= \langle{u,m}\rangle - \rho\langle{u,w}\rangle\\
                        &= k - \rho\lVert u\rVert\\
                        &< k.
  \end{align*}
  But we have now found a $z\in \C$ which contradicts the minimality of the claimed minimizers. It can only be that $\rho = 0$, which implies $s_1 = s_2$, completing the proof of (i).

  To prove (ii), fix unit vectors $w_1, w_2 \in \mathbb{R}^n$ and let $s_1 := \beta(w_1), s_2 := \beta(w_2)$, $d := \lVert s_1 - s_2\rVert$.
  If $d = 0$ the proof is trivial, so suppose $d > 0$.
  For any $t \in [0,1]$, strong convexity of $\C$ implies $B(m_t, r_t) \subseteq \C$, where
  $m_t = (1-t)s_1 + ts_2$ and $r_t = \frac{\alpha}{2}t(1-t)d^2$. Take $z_t := m_t - r_t w_1$, which must lie in $\C$.
  By optimality of $s_1$ for $\min_{s\in \C}\langle{w_1,s}\rangle$,
  \[
    \langle{w_1,s_1}\rangle \leq \langle{w_1, z_t}\rangle = (1-t)\langle{w_1,s_1}\rangle + t\langle{w_1,s_2}\rangle - r_t,
  \]
  which after rearranging gives
  \[
    t\langle{w_1, s_2 - s_1}\rangle \geq r_t = \frac{\alpha}{2}t(1-t)d^2.
  \]
  So for all $t \in (0,1)$,
  \[
    \langle{w_1,s_2 - s_1}\rangle \geq \frac{\alpha}{2}(1-t)d^2.
  \]
  Taking the limit as $t \downarrow 0$, 
  \begin{equation} \label{eq:int-support-l-bound-1}
    \langle{w_1,s_2 - s_1}\rangle \geq \frac{\alpha}{2}d^2.
  \end{equation}
  Swapping $w_1$ and $w_2$ and $s_1$ and $s_2$ in the above working gives the symmetric form
  \begin{equation} \label{eq:int-support-l-bound-2}
    \langle{w_2,s_1 - s_2}\rangle \geq \frac{\alpha}{2}d^2.
  \end{equation}
  Adding \eqref{eq:int-support-l-bound-1} and \eqref{eq:int-support-l-bound-2} we get
  \[
    \langle{w_1 - w_2, s_2 - s_1}\rangle \geq \alpha d^2,
  \]
  and by Cauchy-Schwarz
  \[
    \alpha d^2 \leq \lVert w_1 - w_2\rVert d,
  \]
  which gives the desired inequality
  \[
    \lVert \beta(w_1) - \beta(w_2)\rVert \leq \frac{1}{\alpha}\lVert w_1 - w_2\rVert.
  \]
\end{proof}
\begin{corollary} \label{cor:lmo-lip}
  Suppose that $\C$ is nonempty, compact, and $\alpha$-strongly convex for some $\alpha > 0$. Then for every $u,v \in \R^n\setminus\{0\}$, and every
  $s(u) \in \beta(u)$, $s(v) \in \beta(v)$, one has
  \[
    \|s(u)-s(v)\| \leq \frac{2}{\alpha\,\min(\|u\|,\|v\|)}\,\|u-v\|.
  \]
\end{corollary}
\begin{proof}
  Let $u,v\neq 0$. Since minimizing $\langle u,\cdot\rangle$ is the same as minimizing
  $\langle u/\|u\|,\cdot\rangle$, we may write
  \[
    s(u)=s\!\left(\frac{u}{\|u\|}\right),\quad s(v)=s\!\left(\frac{v}{\|v\|}\right).
  \]
  Hence
  \[
    \|s(u)-s(v)\| \leq \frac{1}{\alpha} \left\|\frac{u}{\|u\|}-\frac{v}{\|v\|}\right\|.
  \]
  Using the bound
  \[
    \left\|\frac{u}{\|u\|}-\frac{v}{\|v\|}\right\| \leq \frac{2}{\min(\|u\|,\|v\|)}\,\|u-v\|,
  \]
  the result follows.
\end{proof}

In their PhD thesis \cite{Hammond1984}, Hammond proved convergence of generalized fictitious play to a solution of \eqref{prob:vip}
under the assumption that $F$ is $C^1$ and monotone, $\C$ is compact and strongly convex, and additionally no point $x \in \C$
satisfies $F(x) = 0$. No rate of convergence is proven in \cite{Hammond1984} for this case. Hammond notes that the assumptions
of strong convexity and $F(x) \neq 0$ for all $x \in \C$ are too restrictive. In Section~\ref{sec:hammond} we showed that neither
of these assumptions are necessary for convergence. It remains unclear how to prove a rate of convergence without the assumption
that $\C$ is strongly convex. However, the assumption that $F$ does not vanish on $\C$ is common in the literature
when proving convergence rates over strongly convex sets. For example, in \cite{Gidel2017}, the authors assume
$\min\left(\|\nabla_x \L(z)\|_{\X^*}, \|\nabla_y \L(z)\|_{\Y^*}\right) \geq \delta > 0$ for all $z \in \X\times\Y$ in order to
obtain global Lipschitz continuity of the LMO over $\X \times \Y$. The resulting linear convergence guarantee of \cite[Theorem~4]{Gidel2017}
depends on $\delta$: as $\delta$ decreases, the Lipschitz modulus worsens and the contraction factor in the rate deteriorates. Hence,
although the rate remains linear in form, the associated complexity bound can become arbitrarily poor when $\delta$ is small.

In the next two theorems, we prove rates of convergence for \eqref{alg:gfp} over strongly convex sets without assuming $F(x)$ does not vanish over $\C$.
The analysis hinges on the observation that when the function value is smaller than some quantity, we can obtain a bound on the Frank-Wolfe gap
in terms of this quantity. This analysis relies on showing $\lVert s_{k+1} - s_k\rVert$ is decreasing at the rate $\O(\gamma_{k+1})$, which we get from
the Lipschitz continuity of the LMO and the operator $F$. In settings where $\C$ is not uniformly smooth,
for example when $\C$ is a polytope, $\lVert s_{k+1} - s_k\rVert$ is not necessarily decreasing and thus another proof technique is necessary.
\begin{figure}[h]
\centering
\begin{minipage}[t]{0.48\textwidth}
\centering
\begin{tikzpicture}[scale=1.0,>=Stealth,line cap=round,line join=round, font=\footnotesize]
  \fill[gray!6] (0,0) rectangle (3.4,3.4);
  \draw[thick] (0,0) rectangle (3.4,3.4);

  \coordinate (x1) at (1.6,1.4);
  \coordinate (x2) at (1.8,1.4);
  \coordinate (s1) at (0,3.4);
  \coordinate (s2) at (3.4,3.4);

  \fill (x1) circle (1.2pt) node[left=1pt] {$x_1$};
  \fill (x2) circle (1.2pt) node[right=1pt] {$x_2$};
  \fill (s1) circle (1.2pt) node[above] {$s_1$};
  \fill (s2) circle (1.2pt) node[above] {$s_2$};

  \draw[->] (x1) -- ++(-0.15,2.0) node[midway, left] {$-F(x_1)$};
  \draw[->] (x2) -- ++(0.15,2.0) node[midway, right] {$-F(x_2)$};


  \draw[decorate,decoration={brace,mirror,amplitude=4pt}] (3.4,3.88) -- (0,3.88)
    node[midway,above=6pt] {$\diam(\C)$};
\end{tikzpicture}

\vspace{0.8ex}
\small (a) $\C$ is a box, where the distance between any two vertices is $\diam(\C)$.
\end{minipage}
\hfill
\begin{minipage}[t]{0.48\textwidth}
\centering
\begin{tikzpicture}[scale=1.0,>=Stealth,line cap=round,line join=round, font=\footnotesize]
  \coordinate (c) at (1.9,1.7);
  \def\r{1.8}
  \fill[gray!6] (c) circle[radius=\r];
  \draw[thick] (c) circle[radius=\r];

  \coordinate (x1) at (1.8,1.5);
  \coordinate (x2) at (2.0,1.5);
  \coordinate (s1) at (1.6,3.5);
  \coordinate (s2) at (2.2,3.5);

  \fill (x1) circle (1.2pt) node[left=1pt] {$x_1$};
  \fill (x2) circle (1.2pt) node[right=1pt] {$x_2$};
  \fill (s1) circle (1.2pt) node[above] {$s_1$};
  \fill (s2) circle (1.2pt) node[above] {$s_2$};

  \draw[->] (x1) -- (s1) node[midway, left] {$-F(x_1)$};
  \draw[->] (x2) -- (s2) node[midway, right] {$-F(x_2)$};


  \draw[decorate,decoration={brace,mirror,amplitude=4pt}] (2.2,3.88) -- (1.6,3.88)
    node[midway,above=2pt] {$\O\big(\frac{1}{m}\|F(x_1) - F(x_2)\|\big)$};
\end{tikzpicture}

\vspace{0.8ex}
\small $\C$ is a ball, which is an example of a strongly-convex set.
\end{minipage}
\caption{Geometry of the linear minimization oracle. On a polytope, crossing the outward normal direction of
a face can cause an $\mathcal{O}(\diam(\C))$ jump in the LMO. For a strongly convex set, the LMO is Lipschitz
on the unit sphere. With $F$ Lipschitz, this allows a bound in terms of the difference between $x_1$ and $x_2$.
Here $m := \min\{\|F(x_1)\|, \|F(x_2)\|\}$.}
\label{fig:lmo-geometry-revised}
\end{figure}

\begin{theorem}\label{thm:mono-rate}
  In addition to the standing assumptions of Section~\ref{sec:hammond}, suppose that $\C$ is $\alpha$-strongly convex.
  Then $F$ is Lipschitz on $\C$ with constant $L > 0$, and for all $k \geq 1$,
  \[
    V(x_{k+1}) \leq \max\big\{(1-\gamma_{k+1})V(x_k) + B\gamma_{k+1}^{3/2}, C\sqrt{\gamma_{k+1}}\big\},
  \]
  where $B = \frac{2L^2\diam(\C)^2}{\alpha}$ and $C = \diam(\C)(1 + L\diam(\C))$. In particular, if $\gamma_k = 1/k$, then
  \[
    V(x_k) \leq \frac{A}{\sqrt{k}},
  \]
  where $A = \max\{V(x_1), 2B, C\}$.
\end{theorem}

\begin{proof}
  Since $F$ is $C^1$ on the compact set $\C$, it is Lipschitz on $\C$. We denote its Lipschitz constant by $L > 0$.
  Observe that $V(x_k) = \langle F(x_k), x_k - s_k\rangle$.

  Since $x_{k+1}=x_k+\gamma_{k+1}(s_k-x_k)$,
  we have
  \[
    x_{k+1}-s_k = (1-\gamma_{k+1})(x_k-s_k).
  \]
  Therefore,
  \begin{align*}
    V(x_{k+1}) &= \langle F(x_{k+1}), x_{k+1}-s_{k+1}\rangle \\
               &= \langle F(x_{k+1}), x_{k+1}-s_k\rangle + \langle F(x_{k+1}), s_k - s_{k+1}\rangle \\
               &= (1-\gamma_{k+1})\langle F(x_{k+1}), x_k - s_k\rangle + \langle F(x_{k+1}), s_k-s_{k+1}\rangle \\
               &= (1-\gamma_{k+1})V(x_k) + (1-\gamma_{k+1})\langle F(x_{k+1}) - F(x_k), x_k - s_k\rangle + \langle F(x_{k+1}), s_k-s_{k+1}\rangle.
  \end{align*}
  Since $x_{k+1}-x_k=\gamma_{k+1}(s_k-x_k)$ and $F$ is monotone,
  \[
    \langle F(x_{k+1}) - F(x_k), s_k - x_k\rangle = \frac{1}{\gamma_{k+1}} \langle F(x_{k+1})-F(x_k), x_{k+1} - x_k\rangle \geq 0.
  \]
  Hence
  \[
    (1-\gamma_{k+1})\langle F(x_{k+1}) - F(x_k), x_k - s_k\rangle \leq 0.
  \]
  Also, by optimality of $s_k \in \beta(F(x_k))$,
  \[
    \langle F(x_k), s_k-s_{k+1}\rangle \leq 0,
  \]
  so
  \begin{align*}
    \langle F(x_{k+1}), s_k-s_{k+1}\rangle &= \langle F(x_{k+1}) - F(x_k), s_k-s_{k+1}\rangle + \langle F(x_k), s_k-s_{k+1}\rangle\\
                                           &\leq \langle F(x_{k+1}) - F(x_k), s_k-s_{k+1}\rangle.
  \end{align*}
  We conclude that
  \[
    V(x_{k+1}) \leq (1-\gamma_{k+1})V(x_k) + \langle F(x_{k+1}) - F(x_k), s_k - s_{k+1}\rangle.
  \]
  Therefore,
  \begin{equation} \label{eqn:recursion-sc-mono}
    V(x_{k+1}) \leq (1-\gamma_{k+1})V(x_k) + \|F(x_{k+1}) - F(x_k)\|\cdot\|s_k-s_{k+1}\|.
  \end{equation}

  Next, define
  \[
    m_k := \min\{\|F(x_k)\|,\|F(x_{k+1})\|\}.
  \]
  If $m_k>0$, Corollary~\ref{cor:lmo-lip} gives
  \[
    \|s_k-s_{k+1}\| \leq \frac{2}{\alpha\,m_k}\,\|F(x_{k+1}) - F(x_k)\|.
  \]
  Also, since $F$ is $L$-Lipschitz on $\C$,
  \[
    \|F(x_{k+1}) - F(x_k)\| \leq L\|x_{k+1}-x_k\| \leq L\gamma_{k+1}\|s_k-x_k\| \leq L\diam(\C)\,\gamma_{k+1}.
  \]
  Substituting into \eqref{eqn:recursion-sc-mono}, we obtain
  \begin{equation}\label{eqn:recursion-sc-mono-2}
    V(x_{k+1}) \leq (1-\gamma_{k+1})V(x_k) + \frac{2L^2\diam(\C)^2}{\alpha\,m_k}\,\gamma_{k+1}^2.
  \end{equation}
  Set $\theta_k := \sqrt{\gamma_{k+1}}$.
  We now split into two cases. In the first case, suppose $m_k > \theta_k$.
  Then \eqref{eqn:recursion-sc-mono-2} yields
  \[
    V(x_{k+1}) \leq (1-\gamma_{k+1})V(x_k) + \frac{2L^2\diam(\C)^2}{\alpha}\,\gamma_{k+1}^{3/2}.
  \]
  Set
  \[
    B := \frac{2L^2\diam(\C)^2}{\alpha}.
  \]
  Then
  \begin{equation}\label{eq:case1-recursion}
    V(x_{k+1}) \leq (1-\gamma_{k+1})V(x_k) + B\,\gamma_{k+1}^{3/2}.
  \end{equation}
  In the second case, suppose $m_k\leq \theta_k$. If $\|F(x_{k+1})\|\leq \theta_k$, then trivially
  \[
    \|F(x_{k+1})\|\leq \theta_k.
  \]
  If instead $\|F(x_k)\|\leq \theta_k$, then
  \[
    \|F(x_{k+1})\| \leq \|F(x_{k+1})-F(x_k)\| + \|F(x_k)\| \leq L\diam(\C)\,\gamma_{k+1}+\theta_k.
  \]
  Thus in either subcase,
  \[
    \|F(x_{k+1})\| \leq L\diam(\C)\,\gamma_{k+1}+\theta_k.
  \]
  Since
  \[
    V(x_{k+1}) = \langle F(x_{k+1}), x_{k+1} - s_{k+1}\rangle \leq \diam(\C)\|F(x_{k+1})\|,
  \]
  we get
  \[
    V(x_{k+1}) \leq \diam(\C)\big(L\diam(\C)\gamma_{k+1} + \sqrt{\gamma_{k+1}}\big).
  \]
  Because $\gamma_{k+1}\leq \sqrt{\gamma_{k+1}}$, it follows that
  \begin{equation}\label{eq:case2-recursion}
    V(x_{k+1}) \leq \diam(\C)(1 + L\diam(\C))\sqrt{\gamma_{k+1}}.
  \end{equation}
  Set
  \[
    C := \diam(\C)(1 + L\diam(\C)).
  \]
  We now prove by induction that
  \[
    V(x_k) \leq \frac{A}{\sqrt{k}},
    \qquad \forall k\geq 1,
  \]
  where
  \[
    A := \max\{V(x_1),2B,C\}.
  \]
  The case $k=1$ is immediate, since $V(x_1) = V(x_1)/\sqrt{1} \leq A/\sqrt{k}$.
  Assume now that $V(x_k) \leq A/\sqrt{k}$ for some $k\geq 1$.

  If the case where $m_k > \theta_k$ holds, then using $\gamma_{k+1}=1/(k+1)$ and \eqref{eq:case1-recursion},
  \[
    V(x_{k+1}) \leq \frac{k}{k+1}\cdot\frac{A}{\sqrt{k}} + \frac{B}{(k+1)^{3/2}}.
  \]
  Also, $\sqrt{k+1}/(\sqrt{k+1} + \sqrt{k}) \geq 1/2$ and $A \geq 2B$,
  \begin{align*}
    \frac{A}{\sqrt{k+1}} - \frac{k}{k+1}\cdot\frac{A}{\sqrt{k}} &= \frac{A}{(k+1)^{3/2}}\left(\frac{\sqrt{k+1}}{\sqrt{k+1}+\sqrt{k}}\right) \\
                                                                     &\geq\frac{A}{2(k+1)^{3/2}}\\
                                                                     &\geq\frac{B}{(k+1)^{3/2}},
  \end{align*}
  Hence, in this case
  \[
    V(x_{k+1}) \leq \frac{A}{\sqrt{k+1}}.
  \]
  If the case where $m_k \leq \theta_k$ holds, then by \eqref{eq:case2-recursion} and the fact that $A \geq C$,
  \[
    V(x_{k+1}) \leq C\sqrt{\gamma_{k+1}} \leq \frac{C}{\sqrt{k+1}} \leq \frac{A}{\sqrt{k+1}}.
  \]
  This completes the induction and proves the theorem.
\end{proof}

When we make the stronger assumption that $F$ is cocoercive, we are able to obtain faster convergence rates.
Note that a $\beta$ cocoercive operator for some $\beta > 0$ is $\frac{1}{\beta}$-Lipschitz continuous.
\begin{theorem}
  In addition to the standing assumptions of Section~\ref{sec:hammond}, suppose that $F$ is $\beta$-cocoercive, with $\beta > 0$ and $\C$ is $\alpha$-strongly convex.
  Assume there exists some $\bar{\gamma} > 0$ such that $1-\gamma_{k+1} \geq \bar{\gamma}$ for all $k \geq 1$.
  Then for all $k \geq 1$,
  \[
    V(x_{k+1}) \leq \max\big\{(1-\gamma_{k+1})V(x_k), B\gamma_{k+1}\big\},
  \]
  where
  \[
    B = \diam(\C)\bigg(\frac{2}{\alpha\beta\bar{\gamma}} + \frac{\diam(\C)}{\beta}\bigg),
  \]
  In particular, if $\gamma_k = 1/k$, then $\bar{\gamma} = \frac{1}{2}$, and
  \[
    V(x_k) \leq \frac{A}{k},
  \]
  where $A = \max\{V(x_1), B\}$.
\end{theorem}
\begin{proof}
  As in the proof of Theorem~\ref{thm:mono-rate},
  \[
    V(x_{k+1}) = (1-\gamma_{k+1})V(x_k) + (1-\gamma_{k+1})\langle F(x_{k+1}) - F(x_k), x_k - s_k\rangle + \langle F(x_{k+1}), s_k-s_{k+1}\rangle.
  \]
  The definition of $\beta$-cocoercivity with $(x,y) = (x_{k+1}, x_k)$ gives
  \begin{align} 
    \langle{F(x_{k+1}) - F(x_k), s_k - x_k}\rangle &= \frac{1}{\gamma_{k+1}}\langle{F(x_{k+1}) - F(x_k), x_{k+1} - x_k}\rangle\notag\\
                                                   &\geq \frac{\beta}{\gamma_{k+1}}\lVert F(x_{k+1}) - F(x_k)\rVert^2.\label{eqn:b-coco-bound}
  \end{align}
  As in the proof of Theorem~\ref{thm:mono-rate}, we control the $\langle{F(x_{k+1}), s_k - s_{k+1}}\rangle$ term by observing
  \[
    \langle{F(x_{k+1}), s_k - s_{k+1}}\rangle \leq \langle{F(x_{k+1}) - F(x_k), s_k - s_{k+1}}\rangle.
  \]
  This time, $\beta$-cocoercivity of $F$ gives us that $F$ is $(1/\beta)$-Lipschitz:
  \begin{equation} \label{eqn:DeltaFk-bound}
    \lVert F(x_{k+1}) - F(x_k)\rVert \leq \frac{1}{\beta}\lVert x_{k+1} - x_k\rVert = \frac{\gamma_{k+1}}{\beta}\lVert s_k - x_k \rVert \leq \frac{\gamma_{k+1}}{\beta}\diam(\C).
  \end{equation}
  Now, define
  \[
    m_k := \min\{\|F(x_k)\|,\|F(x_{k+1})\|\}.
  \]
  If $m_k>0$, Corollary~\ref{cor:lmo-lip} gives
  \[
    \|s_k-s_{k+1}\| \leq \frac{2}{\alpha\,m_k}\,\|F(x_{k+1}) - F(x_k)\|.
  \]
  So, in this case we can combine \eqref{eqn:b-coco-bound} and \eqref{eqn:DeltaFk-bound} to obtain the bound
  \begin{equation}\label{eqn:recursion-sc-coco}
    V(x_{k+1}) \leq (1-\gamma_{k+1})V(x_k) + \bigg(\frac{2}{\alpha\, m_k} - \frac{\beta(1-\gamma_{k+1})}{\gamma_{k+1}}\bigg)\lVert F(x_{k+1}) - F(x_k)\rVert^2.
  \end{equation}
  Set $\theta_k := \frac{2\gamma_{k+1}}{\alpha\beta(1-\gamma_{k+1})}$. Like before, we split into two cases.
  First, suppose we are in the case where $m_k > \theta_k$, then
  \[
    \frac{2}{\alpha\, m_k} - \frac{\beta(1-\gamma_{k+1})}{\gamma_{k+1}} < 0,
  \]
  and
  \begin{equation} \label{eqn:mk-gt}
    V(x_{k+1}) \leq (1-\gamma_{k+1})V(x_k).
  \end{equation}
  In the second case, suppose $m_k \leq \theta_k$. If $\lVert F(x_{k+1})\rVert \leq \theta_k$, then trivially
  \[
    \lVert F(x_{k+1})\rVert \leq \theta_k.
  \]
  If instead $\lVert F(x_k)\rVert \leq \theta_k$, we have
  \[
    \lVert F(x_{k+1})\rVert \leq \frac{\gamma_{k+1}}{\beta}\diam(\C) + \theta_k,
  \]
  as in the proof of Theorem~\ref{thm:mono-rate}.
  In both subcases, we can say that $\lVert F(x_{k+1})\rVert \leq \frac{\gamma_{k+1}}{\beta}\diam(\C) + \theta_k$.
  It follows that
  \begin{align}
    V(x_{k+1}) &= \langle{F(x_{k+1}), x_{k+1} - s_{k+1}}\rangle\notag\\
            &\leq \diam(\C)\lVert F(x_{k+1})\rVert\notag\\
            &\leq \diam(\C)\left(\theta_k + \frac{\gamma_{k+1} \diam(\C)}{\beta}\right)\notag\\
            &= \gamma_{k+1}\diam(\C)\left(\frac{2}{\alpha \beta (1-\gamma_{k+1})} + \frac{\diam(\C)}{\beta}\right)\notag\\
            &\leq \gamma_{k+1} \diam(\C)\left(\frac{2}{\alpha\beta\bar{\gamma}} + \frac{\diam(\C)}{\beta}\right)\notag\\
            &= B\gamma_{k+1},\label{eqn:mk-lt}
  \end{align}
  where we set
  \[
    B := \diam(\C)\left(\frac{2}{\alpha\beta\bar{\gamma}} + \frac{\diam(\C)}{\beta}\right),
  \]
  where we recall $1-\gamma_{k+1} \geq \bar{\gamma} > 0$ for all $k \geq 1$. Therefore, when $m_k \leq \theta_k$,
  \[
    V(x_{k+1}) \leq B\gamma_{k+1}.
  \]
  We now prove by induction that
  \[
    V(x_k) \leq \frac{A}{k},\qquad\forall k \geq 1,
  \]
  where
  \[
    A := \max\big\{V(x_1), B\big\}.
  \]
  The case $k=1$ is clear, since $V(x_k) = V(x_1)/1 \leq A/k$. Assume now that $V(x_k) \leq A/k$ for some $k \geq 1$.
  If the case where $m_k > \theta_k$ holds, then using $\gamma_{k+1} = 1/(k+1)$ and \eqref{eqn:mk-gt},
  \[
    V(x_{k+1}) \leq \frac{k}{k+1}V(x_k) = \frac{k}{k+1}\cdot\frac{A}{k} = \frac{A}{k+1}.
  \]
  If the case where $m_k \leq \theta_k$ holds, then by \eqref{eqn:mk-lt} and the fact that $A \geq B$,
  \[
    V(x_{k+1}) \leq B\gamma_{k+1} = \frac{B}{k+1} \leq \frac{A}{k+1}.
  \]
  This completes the induction and proves the claim.
\end{proof}

\section{Conclusion}
We have shown that the Frank-Wolfe algorithm for solving variational inequalities over compact, convex sets under a monotone
$C^1$ operator and vanishing, nonsummable step sizes converges. We also show iterate convergence to the unique solution
in the special case where $F$ is instead assumed to be strongly monotone. The strongly monotone setting generalizes
Hammond's generalized fictitious play, which was conjectured by Hammond to converge to a solution of the corresponding
variational inequality problem. Thus, this result proves Hammond's longstanding conjecture.

For strongly convex sets, we establish rates of convergence with no assumption that $F$ vanishes over the set $\C$. The convergence
rate of Frank-Wolfe for monotone variational inequality problems over sets that aren't uniformly smooth remains an open problem.
An important subcase of the nonsmooth setting is where $\C$ is a polytope.

\bibliography{refs}

\end{document}